\documentclass{article}

\usepackage{amsmath}
\usepackage{amssymb}
\usepackage{amsthm}
\usepackage{hyperref}
\usepackage{url}

\newtheorem{theorem}{Theorem}[section]
\newtheorem{lemma}[theorem]{Lemma}
\newtheorem{prop}[theorem]{Proposition}

\newtheorem{claim}[theorem]{Claim}

\theoremstyle{definition}
\newtheorem{definition}{Definition}[section]

\theoremstyle{remark}
\newtheorem{remark}{Remark}[section]

\newcommand{\Var}{\operatorname{Var}}

\allowdisplaybreaks

\begin{document}

\title{Littlewood-Offord problems for Ising models}

\author{Yinshan Chang\thanks{Address: College of Mathematics, Sichuan University, Chengdu 610065, China; Email: ychang@scu.edu.cn; Supported by National Natural Science Foundation of China \#11701395.}}

\date{}

\maketitle

\begin{abstract}
We consider the one-dimensional Littlewood-Offord problem for general Ising models. More precisely, we consider the concentration function
\[Q_n(x,v)=P\left(\sum_{i=1}^{n}\varepsilon_iv_i\in(x-1,x+1)\right),\]
where $x\in\mathbb{R}$, $v_1,v_2,\ldots,v_n$ are real numbers such that $|v_1|\geq 1, |v_2|\geq 1,\ldots, |v_n|\geq 1$, and $(\varepsilon_i)_{i=1,2,\ldots,n}\in\{-1,1\}^{n}$ are random spins of some Ising model. Let $Q_n=\sup_{x,v}Q_n(x,v)$. Under natural assumptions, we show that there exists a universal constant $C$ such that for all $n\geq 1$,
\[\binom{n}{[n/2]}2^{-n}\leq Q_n\leq Cn^{-\frac{1}{2}}.\]
As an application of the method, under the same assumption, we give a lower bound on the smallest eigenvalue of the truncated correlation matrix of the Ising model.
\end{abstract}

\section{Introduction}

The classical Littlewood-Offord problem is about a uniform upper bound of the concentration probability
\[P\left(\sum_{i=1}^{n}\varepsilon_iv_i\in(x-1,x+1)\right)\]
where $v_1,v_2,\ldots,v_n$ and $x$ are real numbers such that
\[|v_1|\geq 1, |v_2|\geq 1,\ldots, |v_n|\geq 1,\]
and $(\varepsilon_i)_{i=1,2,\ldots,n}$ are independent Rademacher random variables such that for $i=1,2,\ldots,n$, $P(\varepsilon_i=1)=P(\varepsilon_i=-1)=1/2$. This problem was first considered by Littlewood and Offord \cite{LittlewoodOffordMR0009656} with applications in the study of real roots of random polynomials. Later, Erd\H{o}s \cite{ErdosMR0014608} obtained the following sharp result:
\[\sup_{x\in\mathbb{R}}\sup_{v_1,v_2,\ldots,v_n\in(-\infty,-1]\cup[1,\infty)} P\left(\sum_{i=1}^{n}\varepsilon_iv_i\in(x-1,x+1)\right)=\binom{n}{[n/2]}2^{-n}.\]
The same upper bound was obtained by Kleitman \cite{KleitmanMR0265923} for vectors $x$ and $(v_i)_{i=1,2,\ldots,n}$. We refer to the reference in \cite{TaoVuMR2965282} for the series of work in high dimensions. The inverse problem was first considered by Tao and Vu \cite{TaoVuMR2480613}. It is closely related to random matrices.

Beyond Rademacher series, we notice the results \cite{JKMR4201801} and \cite{SinghalMR4440097} on similar problems for i.i.d. Bernoulli random variables. Beyond independence, we are aware of the result \cite{RaoMR4294326} for symmetric random variables driven by a finite-state reversible stationary Markov chain. Together with Peng, we \cite{LO-CW} considered the Littlewood-Offord problem for spins of Curie-Weiss models and obtained some sharp and asymptotic results. In the present paper, we consider the Littlewood-Offord problems for spins of Ising models.

Ising models are models for magnets in statistical physics. Let us briefly introduce Ising models in the following. Consider a graph $G=(V,\mathcal{E})$ with the vertex set $V$ and the edge set $\mathcal{E}$. Let $V_n\subset V$ be a subset of vertices with $n$ vertices. Let $\mathcal{E}_{n}$ be the set of edges adjacent to some vertex in $V_n$, i.e.
\[\mathcal{E}_n=\{\{i,j\}\in\mathcal{E}:i\in V_n\text{ or }j\in V_n\}.\]
Let $\Sigma_n=\{-1,1\}^{V_n}$ be the configuration space of spins. Let $\kappa\in\{-1,1\}^{V\setminus V_n}$ be a boundary condition. Let $h\in\mathbb{R}^{V_n}$ be the external field. For a configuration $\sigma\in\Sigma_n$, we extend $\sigma$ to a boolean function on $V$ by taking $\sigma_j=\kappa_j$ for $j\notin V_n$. We define the energy of $\sigma$ via a function called Hamiltonian as follows:
\begin{equation}
 H_{n,J,h,\kappa}(\sigma) = -\sum_{\{i,j\}\in\mathcal{E}_n}J_{ij}1_{\sigma_i=\sigma_j}-\sum_{i\in V_n}h_i\sigma_i,
\end{equation}
where $J_{ij}=J_{ji}\in\mathbb{R}$ is the coupling constant. Without loss of generality, we always assume that
\[J_{ij}\neq 0,\forall \{i,j\}\in\mathcal{E}.\]

\begin{definition}
The \emph{Gibbs distribution} $\mu_n$ of the Ising model is a probability measure on $\{-1,1\}^{V_n}$, which is defined by
\begin{equation}
\mu_{n}(\sigma)=e^{-H_{n,J,h,\kappa}(\sigma)}/Z_{n},
\end{equation}
where $Z_{n}$ is the normalizing constant.
\end{definition}
Let $(\varepsilon_i)_{i=1,2,\ldots,n}$ be some random vector sampled according to the Gibbs distribution $\mu_{n}$. We are interested in
\begin{equation}\label{eq: defn Qn}
Q_n=\sup_{x\in\mathbb{R}}\sup_{|v_1|,|v_2|,\ldots,|v_n|\geq 1}P\left(\sum_{i=1}^{n}\varepsilon_iv_i\in(x-1,x+1)\right).
\end{equation}
Our main result is the upper bound and the lower bound of $Q_n$ under the boundedness assumption on the coupling constants $(J_{ij})_{i,j}$ and the external field $(h_i)_i$.
\begin{theorem}\label{thm: main}
Assume that
\begin{equation}\label{eq: boundedness assumption}
K=\sup_{n\geq 1}\max_{i\in V_n}\sum_{j\in V}|J_{ij}|+2\max_{i\in V_n}|h_i|<+\infty.
\end{equation}
Then, there exists a universal constant $C(K)$ such that
\begin{equation}\label{eq: lower and upper bounds for Qn}
\binom{n}{[n/2]}2^{-n}\leq Q_n\leq C(K)n^{-\frac{1}{2}}.
\end{equation}
\end{theorem}
So, under the boundedness assumption \eqref{eq: boundedness assumption}, $Q_n$ is of the order $1/\sqrt{n}$.
\begin{remark}
The lower bound $\binom{n}{[n/2]}2^{-n}\leq Q_n$ in \eqref{eq: lower and upper bounds for Qn} is quite general and holds for all random vectors $(\varepsilon_i)_i\in\{-1,1\}^n$, see Theorem~\ref{thm: general lower bound}.
\end{remark}

\emph{Organization of the paper}: In Section~\ref{sect: preliminary}, we explain Edwards-Sokal coupling for Ising models. The difference from the classical paper \cite{EdwardsSokalMR0965465} is that we allow negative coupling constants $J_{ij}$. In Section~\ref{sect: general lower bounds}, we prove a general lower bound for $Q_n$, which holds for all boolean random vectors. In Section~\ref{sect: Ising}, we prove the upper bound in Theorem~\ref{thm: main}. In Section~\ref{sect: applications}, as an application of our method, we give a lower bound on the smallest eigenvalue of the covariance matrix of the spins in Ising models, see Proposition~\ref{prop: covariance matrix}.

\section{Preliminary: the Edwards-Sokal coupling for Ising models}\label{sect: preliminary}

In this section, we will explain the Edwards-Sokal coupling \cite{EdwardsSokalMR0965465} for Ising models beyond ferromagnetism. For simplicity of notation, we consider the case without external field $h$. The presence of an external field $h$ could be transformed to a certain boundary condition. Indeed, let $\widetilde{V}_n$ be a copy of $V_n$, where the vertex $\widetilde{i}\in\widetilde{V}_n$ is a copy of the vertex $i\in V_n$. Let $\widetilde{V}=V\cup\widetilde{V}_n$ be an enlargement of the original vertex set $V$. We enlarge the edge set $\mathcal{E}$ to $\widetilde{\mathcal{E}}$ by adding the edges $\{i,\widetilde{i}\}$ between $V_n$ and $\widetilde{V}_n$. We put no edge between vertices of $\widetilde{V}_n$. We set the coupling constant
\[J_{i\widetilde{i}}=2h_i\]
for $i\in V_n$ and $\widetilde{i}\in\widetilde{V}_n$. We extend the boundary condition $\kappa$ to $\widetilde{\kappa}$ by setting $\widetilde{\kappa}_{\widetilde{i}}=1$ for all $\widetilde{i}\in\widetilde{V}_n$. Then, the original Gibbs distribution is the same as the Gibbs distribution on the extended graph with boundary condition $\widetilde{\kappa}$ and vanishing external field. Thus, without loss of generality, we assume that there is no external field.

Define
\begin{equation*}
\mathcal{E}_{n}^{+}=\{\{i,j\}\in\mathcal{E}_n:J_{ij}>0\},\quad\mathcal{E}_{n}^{-}=\{\{i,j\}\in\mathcal{E}_n:J_{ij}<0\},
\end{equation*}
i.e., $\mathcal{E}_{n}^{+}$ (resp. $\mathcal{E}_{n}^{-}$) is the collection of edges adjacent to $V_n$ with positive (resp. negative) coupling constants. Recall that we assume that $J_{ij}\neq 0$ if $\{i,j\}$ is an edge in $\mathcal{E}$. Hence, $\mathcal{E}_n$ is the disjoint union of $\mathcal{E}_{n}^{+}$ and $\mathcal{E}_{n}^{-}$. We define a model with coupled random spins and random edges as follows: Let $\sigma=(\sigma_i)_i\in\{-1,1\}^{V_n}$ be the configuration of spins on vertices and let $e=(e_{ij})_{\{i,j\}}\in\{0,1\}^{\mathcal{E}_n}$ be the edge configuration. (Here, the edge $\{i,j\}$ is open if and only if $e_{ij}=1$. The edge $\{i,j\}$ is closed if it is not open.) Define the probability $\mu_{FKSW}(\sigma,e)$ that $\sigma$ and $e$ are realized by
\begin{align}
\mu_{FKSW}(\sigma,e)=&Z_{FKSW}^{-1}\prod_{\{i,j\}\in \mathcal{E}_{n}^{+}}((1-p_{ij})1_{e_{ij}=0}+p_{ij}1_{e_{ij}=1}1_{\sigma_i=\sigma_j})\notag\\ &\times\prod_{\{i,j\}\in\mathcal{E}_{n}^{-}}((1-p_{ij})1_{e_{ij}=0}+p_{ij}1_{e_{ij}=1}1_{\sigma_i\neq\sigma_j}),
\end{align}
where $p_{ij}=1-e^{-|J_{ij}|}$ and $Z_{FKSW}$ is the normalizing constant. Note that $\mu_{FKSW}(\sigma,e)>0$ iff the following constraint is satisfied: For all $\{i,j\}\in\mathcal{E}_{n}^{+}$, $e_{ij}$ must be $0$ as long as $\sigma_i\neq\sigma_j$; and for all $\{i,j\}\in\mathcal{E}_{n}^{-}$, $e_{ij}$ must be $0$ as long as $\sigma_i=\sigma_j$. The marginal distribution of the spin configuration $\sigma$ is given by
\begin{equation}
\mu_{\mathrm{Ising}}(\sigma)=Z^{-1}_{\mathrm{Ising}}e^{\sum_{\{i,j\}\in\mathcal{E}_{n}^{+}}J_{ij}(1_{\sigma_i=\sigma_j}-1)} e^{\sum_{\{i,j\}\in\mathcal{E}_{n}^{-}}J_{ij}1_{\sigma_i=\sigma_j}},
\end{equation}
which is proportional to
\[e^{\sum_{\{i,j\}\in\mathcal{E}_{n}}J_{ij}1_{\sigma_i=\sigma_j}}.\]
Hence, the marginal distribution $\mu_{\mathrm{Ising}}$ is the same as the Gibbs measure of the Ising model (without external field). The marginal distribution of the edge configuration will be called the random cluster model. Due to the possible presence of negative coupling constants $J_{ij}$, the description of the random cluster model is more complicated. As we don't need the precise definition of the random cluster model, we will not discuss the details here.

The two conditional distributions $\mu_{FKSW}(e|\sigma)$ and $\mu_{FKSW}(\sigma|e)$ are important to us. We will discuss them in the following.

Given the spin configuration $\sigma$, the edge random variables $(e_{ij})_{ij}$ are conditionally independent. Moreover, for $\{i,j\}\in\mathcal{E}_{n}^{+}$, when $\sigma_i\neq\sigma_j$, we must have $e_{ij}=0$; when $\sigma_i=\sigma_j$, the conditional probability $\mu_{FKSW}(e_{ij}=1|\sigma)=p_{ij}$. For $\{i,j\}\in\mathcal{E}_{n}^{-}$, when $\sigma_i=\sigma_j$, we must have $e_{ij}=0$; when $\sigma_i\neq\sigma_j$, the conditional probability $\mu_{FKSW}(e_{ij}=1|\sigma)=p_{ij}$. In summary, given the spin configuration $\sigma$, the process $(e_{ij})_{ij}$ is a Bernoulli bond percolation.

Given an edge configuration $e$ with strictly positive possibility, we first define clusters as connected components of $\overline{V_n}$ under certain equivalence relation $\sim$ on $\overline{V_n}$, where
\[\overline{V_n}=\{i\in V: i\in V_n\text{ or }\exists j\in V_n\text{ such that }\{i,j\}\in\mathcal{E}_n\}.\]
Hence, $\overline{V_n}$ is the set $V_n$ plus its neighbor vertices. For two different vertices $i_0$ and $i_{m}$ in $\overline{V_n}$, we define $i_0\sim i_m$ iff there exists $i_1,i_2,\ldots,i_{m-1}\in\overline{V_n}$ such that $e_{i_0i_1}=e_{i_1i_2}=\cdots=e_{i_{m-1}i_m}=1$. Let $C_1,C_2,\ldots,C_{\ell},C_{\ell+1},\ldots,C_{\ell+q}$ be the clusters. Here, without loss of generality, we assume that $C_1,C_2,\ldots,C_{\ell}$ are contained in $V_n$, and $C_{\ell+1},\ldots,C_{\ell+q}$ are not entirely contained in $V_n$. To distinguish these two kinds of clusters, we call the former inner clusters and the later boundary clusters. Fix a total order on $V$. For each $k\geq 1$, choose the smallest vertex $c_k$ in the cluster $C_k$. Define $\eta_k=\sigma_{c_k}$ for $k\geq 1$. For each $k$, once the spin $\eta_k$ is fixed, the spins of the other vertices in $C_k$ are fixed at the same time, according to the following rule: If $\{i,j\}\in\mathcal{E}_{n}^{+}$ and $e_{ij}=1$, then $\sigma_i=\sigma_j$; if $\{i,j\}\in\mathcal{E}_{n}^{-}$ and $e_{ij}=1$, then $\sigma_i=-\sigma_j$. For similar reasons, the spins of vertices in boundary clusters are determined by the boundary condition. Note that there is no conflict of spins since the edge configuration occurs with strictly positive probability under $\mu_{FKSW}$. Next, we call $\eta_k$ the cluster-spin of $C_k$. For a vertex $i$ in a cluster $C_k$, we define
\begin{equation}
 S(i)=\eta_k\sigma_i.
\end{equation}
(Note that we always have $S(c_k)=1$.) In this way, we define a boolean function $S$ on $\overline{V_n}$. Although $S$ depends on $e$ and $\sigma$ by definition, it actually depends only on $e$ as we fix $S(c_k)=1$. Finally, given the edge configuration $e$, the cluster-spins $\eta_{\ell+1},\ldots,\eta_{\ell+q}$ of the boundary clusters are determined by the boundary condition, and the cluster-spins $\eta_1,\eta_2,\ldots,\eta_{\ell}$ of the inner clusters are conditionally independent with the common distribution $\mu_{FKSW}(\eta_j=1|e)=\mu_{FKSW}(\eta_j=0|e)=1/2$. In summary, the spin $\sigma_i$ at the vertex $i$ is fixed according to $S(i)$ and its cluster-spin $\eta_k$. Given the edge configuration $e$, the function $S$ is fixed, the cluster-spins of the inner clusters form Rademacher series, and the cluster-spins of the boundary clusters are determined by boundary conditions.

\section{General lower bounds}\label{sect: general lower bounds}

In this section, we obtain a lower bound for $Q_n$ that holds for any random vector taking values in $\{-1,1\}^n$.
\begin{theorem}\label{thm: general lower bound}
Consider an arbitrary random vector $\varepsilon=(\varepsilon_1,\varepsilon_2,\ldots,\varepsilon_n)$ taking values in $\{-1,1\}^n$. As before, we define
\begin{equation*}
Q_n=\sup_{x\in\mathbb{R}}\sup_{|v_1|,|v_2|,\ldots,|v_n|\geq 1}P(\varepsilon_1v_1+\varepsilon_2v_2+\cdots+\varepsilon_nv_n\in(x-1,x+1)).
\end{equation*}
Then, for all $n\geq 1$, we have that
\begin{equation}
 Q_n\geq \binom{n}{[n/2]}2^{-n}.
\end{equation}
\end{theorem}

\begin{proof}
Let $I_x=(x-1,x+1)$ and
\[Q_n(x,v)=P\left(\sigma\in \{-1,1\}^n:\, \sum_{i=1}^n v_i\,\sigma_i\in I_x\right).\]
Then
\[Q_n\geq\sup_{x\in\mathbb{R}}\sup_{v\in \{-1,1\}^n} Q_n(x,v).\]
Now, let $B$ be the uniform measure over $\{-1,1\}^n$. For each $\sigma\in \{-1,1\}^n$, we have that
\begin{align}\label{eq: B v 1}
\sum_{v\in\{-1,1\}^n}B(v)1_{I_x}(\sum_{i=1}^n v_i\,\sigma_i)\overset{\delta_i=v_i\sigma_i}{=}&\sum_{\delta\in\{-1,1\}^d}B(\delta)1_{I_x}(\sum_{i=1}^n\delta_i)\notag\\
=&B\left(v\in\{-1,1\}^n:\,\sum_{i=1}^n v_i\in I_x\right),
\end{align}
where $1_{A}(x)$ is the indicator function of the set $A$. Since
\begin{align*}
\sup_{v\in \{-1,1\}^n} Q_n(x,v) \geq &\sum_{v\in\{-1,1\}^n}B(v)Q_n(x,v)\\
=&\sum_{v\in \{-1,1\}^n}\sum_{\sigma\in\{-1,1\}^n}B(v)P(\sigma)1_{I_x}(\sum_{i=1}^nv_i\,\sigma_i)\\
=&\sum_{\sigma\in\{-1,1\}^n}P(\sigma)\sum_{v\in \{-1,1\}^n} B(v)1_{I_x}(\sum_{i=1}^nv_i\,\sigma_i)\\
\overset{\eqref{eq: B v 1}}{=}&\sum_{\sigma\in\{-1,1\}^n}P(\sigma)B\left(v\in\{-1,1\}^n:\,\sum_{i=1}^n v_i\in I_x\right)\\
=&B\left(v\in\{-1,1\}^n:\,\sum_{i=1}^n v_i\in I_x\right),
\end{align*}
we get that
\[Q_n\geq \sup_{x\in\mathbb{R}}B\left(v\in\{-1,1\}^n:\,\sum_{i=1}^n v_i\in I_x\right)=\binom{n}{[n/2]}.\]
\end{proof}

\section{Upper bounds}\label{sect: Ising}

In this section, we prove the upper bound in Theorem~\ref{thm: main}.

As explained at the beginning of Section~\ref{sect: preliminary}, without loss of generality, we assume that the external field $h$ vanishes. Recall the Edwards-Sokal coupling and the notation in Section~\ref{sect: preliminary}.

Firstly, we give a sketch of the proof. By the Edwards-Sokal coupling, given the edge configuration, the random spins $\eta_k$ of the inner clusters form a sequence of i.i.d. Bernoulli random variables with parameter $1/2$. In this way, we turn the problem for Ising models into the classical Littlewood-Offord problem for Rademacher series, see Lemma~\ref{lem: from dependent Ising to Rademacher series} below. Then we wish to use the classical bound for i.i.d. symmetric Bernoulli variables $\eta_k$. However, we encounter a problem there. The coefficient $w_k$ before $\eta_k$ is a signed sum of the coefficients $v_j$. Due to the possible cancellation, it may happen that $|w_k|<1$, which prevents us from using the classical result. Fortunately, for the $k$-th cluster containing only a single vertex $j$, $|w_k|=|v_j|\geq 1$. Let us call these spins isolated spins. By conditioning on the non-isolated spins and using the classical Littlewood-Offord bound for isolated spins, we obtain an upper bound of $Q_n$ in terms of the number $N$ of isolated vertices, see \eqref{eq: upper bound via N}. In the final step, we prove that $N$ is linear in $n$ with high probability, and the upper bound follows, see Lemma~\ref{lem: lower bounds on the number of clusters}.

The whole proof is separated into three steps.

\textbf{Step 1}: from Ising to Rademacher

\begin{lemma}\label{lem: from dependent Ising to Rademacher series}
 The random variable $\sum_{i=1}^{n}\varepsilon_iv_i$ is equal to the random variable $\sum_{k=1}^{\ell}\eta_kw_k+a$,
 where for $k=1,2,\ldots,\ell+q$, $w_{k}:=\sum_{j\in C_k}S(j)v_j$ and $a:=\sum_{k=\ell+1}^{\ell+q}\eta_{k}w_k=\sum_{k=\ell+1}^{\ell+q}\sum_{j\in C_k}\varepsilon_jv_j$. In particular,
 \begin{multline*}
 P\left(\sum_{i=1}^{n}\varepsilon_iv_i\in(x-1,x+1)\right)\\
 =\mu_{FKSW}\left(\sum_{k=1}^{\ell}\eta_kw_k=(x-a-1,x-a+1)\right).
 \end{multline*}
\end{lemma}
Fix $(v_j)_{j}$ and the boundary condition of the Ising model. Note that the number $\ell$ of inner clusters and $w_1,w_2,\ldots,w_{\ell}$ and $a$ are determined by the edge configuration $e$. Given the edge configuration $e$, the cluster-spins $\eta_1,\eta_2,\ldots,\eta_{\ell}$ are i.i.d. Bernoulli random variables with parameter $1/2$.

\textbf{Step 2}: upper bound via the number $N$ of isolated vertices

By the Littlewood-Offord bound \cite[Theorem~1]{ErdosMR0014608} for Rademacher series, if $|w_1|,|w_2|,\ldots,|w_{\ell}|\geq 1$, then there exists a finite universal constant $C_1$ such that for all $x\in\mathbb{R}$,
\[\mu_{FKSW}\left(\left.\sum_{k=1}^{\ell}\eta_kw_k\in(x-a-1,x-a+1)\right|e\right)\leq C_1/\sqrt{\ell+1}.\]
However, although $|v_1|,|v_2|,\ldots,|v_n|\geq 1$, we don't necessarily have
\[|w_k|\geq 1,\forall k=1,2,\ldots,\ell\]
in general. Fortunately, if the inner cluster $C_k$ consists of a single vertex, we do have $|w_k|\geq 1$ as $|w_k|$ is equal to some $|v_{j}|$. By reordering the clusters if necessary, without loss of generality, we assume that $C_1,C_2,\ldots,C_{N}$ are all the inner clusters formed by a single vertex, where $N\leq\ell$ is the number of isolated (inner) vertices in the graph with the vertex set $V_n$ and the edge configuration $e$. By the classical Littlewood-Offord theorem for Rademacher series, there exists a finite universal constant $C_1$ such that for all $y\in\mathbb{R}$,
 \[\mu_{FKSW}\left(\left.\sum_{k=1}^{N}\eta_kw_k\in(y-1,y+1)\right|e\right)\leq C_1/\sqrt{N+1}.\]
 Conditionally on the edge configuration $e$, $(\eta_1,\eta_2,\ldots,\eta_{N})$ is independent from $(\eta_{N+1},\eta_{N+2},\ldots,\eta_{\ell})$. Hence, we have that
 \begin{align*}
 &\mu_{FKSW}\left(\left.\sum_{k=1}^{\ell}\eta_kw_k\in(x-a-1,x-a+1)\right|e\right)\\
 =&\mu_{FKSW}\left[\mu_{FKSW}\left(\left.\left.\sum_{k=1}^{N}\eta_{k}w_{k}\in(y-1,y+1)\right|e,\eta_{N+1},\ldots,\eta_{\ell}\right)\right|e\right]\\
 \leq & E[C_1/\sqrt{N+1}|e]\\
 \leq & C_1/\sqrt{N+1},
 \end{align*}
 where $y=x-a-\sum_{k>N}\eta_{k}w_k$. Hence, by taking the expectation and using Lemma~\ref{lem: from dependent Ising to Rademacher series}, we have that
 \begin{equation}\label{eq: upper bound via N}
 Q_n\leq E[C_1/\sqrt{N+1}].
 \end{equation}

 \textbf{Step 3}: the linear growth of $N$ and the conclusion

 With high probability, $N$ is of the same order as $n$. More precisely, we have the following result.
 \begin{lemma}\label{lem: lower bounds on the number of clusters}
 Recall \eqref{eq: boundedness assumption}. Then, there exist universal constants $c_1=c_1(K)$ and $C_2=C_2(K)$ such that for all $n$,
 \[P(N\geq c_1n)\geq 1-C_2n^{-10}.\]
 \end{lemma}
 Combining previous results, we have that
 \[Q_n\leq C_1/\sqrt{c_1n+1}+P(N<c_1 n)\leq C_2n^{-10}+C_1/\sqrt{c_1n}\leq C_3/\sqrt{n},\]
 where $C_3=C_3(K)<\infty$ is a universal constant.

 Finally, it remains to prove Lemma~\ref{lem: lower bounds on the number of clusters}.

 \begin{proof}[Proof of Lemma~\ref{lem: lower bounds on the number of clusters}]
 We first sample the spin configuration $\sigma$ according to the Gibbs measure $\mu_{n}$ of the Ising model. Then, we sample the edge configuration $e$ as follows: For an edge $\{i,j\}\in\mathcal{E}_{n}^{+}$, if $\sigma_i\neq\sigma_j$, then set $e_{ij}=0$; if $\sigma_i=\sigma_j$, then set $e_{ij}=0$ with conditional probability $1-p_{ij}$, conditionally independent from the other edges. For an edge $\{i,j\}\in\mathcal{E}_{n}^{-}$, if $\sigma_i=\sigma_j$, then set $e_{ij}=0$; if $\sigma_i\neq\sigma_j$, then set $e_{ij}=0$ with conditional probability $1-p_{ij}$, conditionally independent from the other edges. By Edwards-Sokal coupling explained in Section~\ref{sect: preliminary}, the joint distribution of $(\sigma,e)$ is precisely $\mu_{FKSW}$. Define
 \[V_n^{+}=\{i\in V_n:\sigma_i=1\},\quad V_n^{-}=\{i\in V_n:\sigma_i=-1\}.\]
 Then, $\{V_n^{+},V_n^{-}\}$ is a partition of $V_n$ and $\max(|V_n^{+}|,|V_n^{-}|)\geq n/2$. Without loss of generality, we may assume $|V_n^{+}|\geq n/2$. Consider the conditioned bond percolation process. Let $N^{+}$ be the number of isolated vertices in $V_n^{+}$. It suffices to prove Lemma~\ref{lem: lower bounds on the number of clusters} for $N^{+}$ instead of $N$. For this purpose, we consider the following percolation process on the graph $G=(W,\mathcal{E})$. Here, $W=\overline{V_n}$ and
 \[\mathcal{E}=\{\{i,j\}\in\mathcal{E}_n^{+}:\sigma_i=\sigma_j\}\cup\{\{i,j\}\in\mathcal{E}_n^{-}:\sigma_i\neq\sigma_j\}.\]
 Let $V$ be a subset of vertices of $W$. Assume that $|V|=m\geq n/2$. Here, $V=V_{n}^{+}$. For the probability $p_{ij}$ that $\{i,j\}\in\mathcal{E}$ is open, we assume that there exists a universal constant $K<\infty$ such that for all $i\in V$,
 \begin{equation}\label{eq: upper bounds K sum proba}
 \sum_{j:\{i,j\}\in\mathcal{E}}-\ln(1-p_{ij})\leq K,
 \end{equation}
 which is guaranteed by \eqref{eq: boundedness assumption} and $p_{ij}=1-e^{-|J_{ij}|}$. Let $M$ be the number of isolated vertices in $V$. We have the following claim for general bond percolation under the assumption \eqref{eq: upper bounds K sum proba}:
 \begin{claim}
 There exist universal constants $c=c(K)$ and $C=C(K)$ such that
 \begin{equation}\label{eq: percolation lower bounds number of clusters}
 P(M>c m)\geq 1-Cm^{-10}.
 \end{equation}
 For a more precise statement, see \eqref{eq: P(M>=delta m)} below.
 \end{claim}

 To prove \eqref{eq: percolation lower bounds number of clusters}, we gradually discover all the isolated vertices in $V$ by the following algorithm. We list the vertex set $V$ as $i_1,i_2,\ldots,i_m$. And we fix an order on the set of edges. Let $V_1=V$ and $v_1=i_1$. Next, we sample the edges adjacent to $i_1$ in order according to Bernoulli random variables with parameter $p_{i_1j}$, i.e., we keep the edge $\{i_1,j\}$ with probability $p_{i_1j}$ and remove the edge $\{i_1,j\}$ with probability $1-p_{i_1j}$. And there is independence among different edges. Define
 \[U_1=\{j\in W:\{i_1,j\}\text{ is not removed}\}.\]
 Note that $i_1$ is isolated iff $U_1=\emptyset$. If $U_1\neq\emptyset$, then we immediately know that the vertices in $U_1\cup\{i_1\}$ are not isolated. In any case, we don't need to examine the status of the vertices in $U_1\cup\{i_1\}$ anymore in the following steps. We will recursively define $V_{k}$, $v_k\in V_k$ and $U_k$: The set $V_k$ is the unexplored region of $V$ after the first $k-1$ steps. In the $k$-th step, we choose some vertex $v_k\in V_k$ and examine the remaining edges adjacent to $v_k$. Roughly speaking, the vertices in $U_k$ are non-isolated vertices discovered in the $k$-th step. Suppose $V_1,V_2,\ldots,V_k,v_1,\ldots,v_{k},U_1,U_2,\ldots,U_k$ are well defined. We set $V_{k+1}=V_{k}\setminus(U_k\cup\{v_k\})$. If $V_{k+1}=\emptyset$, then the algorithm actually finishes and we set $U_{k+1}=V_{k+2}=U_{k+2}=V_{k+3}=\cdots=\emptyset$. If $V_{k+1}\neq\emptyset$, we choose the vertex $i_{f(k+1)}\in V_{k+1}$ with the smallest index $f(k+1)$ and define $v_{k+1}=i_{f(k+1)}$. Then, we sample the remaining edges adjacent to $v_{k+1}$ in order. We remove an edge $e$ with probability $1-p_{e}$. It is possible to sample an edge with the other endpoint in $\cup_{j\leq k}U_j$. (Note that the edges $\{v_j,v_{k+1}\}$ are already removed for $j=1,2,\ldots,k$. Otherwise, $v_{k+1}$ will belong to some $U_{j}$ and it will not belong to the unexplored region $V_{k+1}$.) Define
 \[U_{k+1}=\{j\in W:\{v_{k+1},j\}\text{ is not removed}\}.\]
 Note that $v_{k+1}$ is isolated iff $U_{k+1}=\emptyset$. Let
 \[T=\max\{k\geq 1:V_{k}\neq\emptyset\}.\]
 We say that the algorithm lasts for $T$ turns. By the construction of the algorithm, we have the following expression for the number $M$ of isolated vertices:
 \[M=\sum_{k=1}^{T}1_{U_k=\emptyset}.\]
 Since we may discover the subset of isolated vertices in $V$ by this algorithm, we will call such an algorithm the algorithm to search (all) the isolated vertices. In this algorithm, we sample the edges in random order. We need to show that the edge configuration in our algorithm has the same distribution as the distribution in the bond percolation. There are two main ingredients: the independence among edges in the bond percolation and the measurability of the exploration process. To be more precise, we list the edges in the order of exploration. Let $\mathcal{S}_n$ be the set of the first $n$ edges. Then the next edge $e_{n+1}$ to explore is a deterministic function $e(\mathcal{S}_n)$ of $\mathcal{S}_n$. Moreover, for a deterministic set $S$ of edges, the event $\{\mathcal{S}_n=S\}$ is measurable with respect to the edge configuration on $S$. (This property is similar to the definition of ``stopping times''.) Hence,
 \[P_{\textrm{percolation}}(e_{n+1}\text{ is open}|\mathcal{S}_n\text{ and the edge configuration on }\mathcal{S}_n)=p_{e(\mathcal{S}_n)},\]
 which agrees with our algorithm. Hence, our algorithm can be used to count the number of isolated vertices in a bond percolation model.

 Define
 \[X_k=|U_k|,\mathcal{F}_k=\sigma(V_j,v_j,U_j,j\leq k)\text{ for }k\geq 1.\]
 Define $\mathcal{F}_0=\{\emptyset,\Omega\}$ to be the trivial minimal $\sigma$-field. Then, we have that
 \[\sum_{k=1}^{T}(X_k+1)\geq m=|V|.\]
 We have an inequality here for two reasons: Firstly, $U_k$ may contain vertices in $W\setminus V$; secondly, there may exist overlaps between different $U_k$. For a Bernoulli random variable $\xi$ such that $P(\xi=1)=1-P(\xi=0)=p$, $\xi$ is stochastically dominated by a Poisson random variable with parameter $\lambda=-\ln(1-p)$. Besides, the sum of independent Poisson random variables is still a Poisson random variable with the new parameter as the sum of parameters. Hence, conditionally on $\mathcal{F}_k$, $X_k$ is stochastically dominated by a Poisson random variable with parameter $K$, where $K$ appears in the condition \eqref{eq: upper bounds K sum proba}. Let $(Y_k)_{k\geq 1}$ be i.i.d. Poisson random variables with the same parameter $K$. Note that $T$ stochastically dominates $T^{Y}$, where
 \[T^{Y}=\inf\{k\geq 1:Y_1+Y_2+\cdots+Y_k\geq m-k\}.\]
 A key observation is that
 \begin{equation}\label{eq: key observation}
 M\text{ stochastically dominates }M^{Y}=\sum_{k=1}^{T^{Y}}1_{Y_k=0}.
 \end{equation}
 By standard arguments with exponential Markov inequality, for a Poisson random variable $Z$ with parameter $\lambda$,
 \begin{equation}\label{eq: Poisson tail}
 P(Z\geq 2\lambda)\leq \exp(-(2\ln 2-1)\lambda).
 \end{equation}
 For $c=c(K)=1/(2K+1)$, we take $k=[cm]$. Then $Y_1+Y_2+\cdots+Y_k$ is a Poisson random variable with parameter $kK$. Moreover, we have that
 \begin{align}\label{eq: lower tail for T^Y}
 P(T^{Y}\leq cm)&\leq P(Y_1+Y_2+\cdots+Y_k\geq m-k)\notag\\
 &\leq P(Z\geq (1-c)m)\notag\\
 &\overset{\lambda=cKm}{=}P(Z\geq 2\lambda)\notag\\
 &\overset{\eqref{eq: Poisson tail}}{\leq}\exp(-(2\ln 2-1)\lambda)\notag\\
 &= \exp(-(2\ln 2-1)Km/(2K+1)),
 \end{align}
 where $Z$ is a Poisson random variable with parameter $\lambda=cKm$. By Hoeffding's inequality, for a binomial random variable $W$ with parameter $(k,q)$, we have that
 \begin{equation}\label{eq: Binomial tail}
 P(W\leq kq/2)\leq e^{-kq^2/2}.
 \end{equation}
 For $k=[cm]$ with $c=1/(2K+1)$, $\sum_{j=1}^{k+1}1_{Y_j=0}$ is a binomial random variable with parameter $(k+1,q)$, where $q=e^{-K}$. Take $\delta=\delta(K)=cq/2=e^{-K}/(4K+2)$. Then $\delta m=cmq/2\leq (k+1)q/2$. Therefore, we have that
 \begin{align*}
 P(M^Y\leq \delta m) &\leq P(T^Y\leq cm)+P\left(T^Y>cm,\sum_{j=1}^{k+1}1_{Y_j=0}\leq\delta m\right)\\
 &\leq P(T^Y\leq cm)+P\left(\sum_{j=1}^{k+1}1_{Y_j=0}\leq (k+1)q/2\right)\\
 &\overset{\eqref{eq: lower tail for T^Y},\eqref{eq: Binomial tail}}{\leq} \exp(-(2\ln 2-1)Km/(2K+1))+\exp(-(k+1)q^2/2)\\
 &\leq \exp(-(2\ln 2-1)Km/(2K+1))\\
 &\quad+\exp(-e^{-2K}m/(4K+2)).
 \end{align*}
 Finally, by the key observation \eqref{eq: key observation}, for $\delta=e^{-K}/(4K+2)$, we have that
 \begin{align}\label{eq: P(M>=delta m)}
 P(M\geq \delta m)&\overset{M^Y\leq_{s.t.} M}{\geq} P(M^Y\geq \delta m)\notag\\
 &\geq 1-\exp(-(2\ln 2-1)Km/(2K+1))\notag\\
 &\quad-\exp(-e^{-2K}m/(4K+2))
 \end{align}
 and \eqref{eq: percolation lower bounds number of clusters} is proved.
 \end{proof}

\section{Applications}\label{sect: applications}

The Littlewood-Offord inequality
\begin{equation}\label{eq: LO ineq}
\sup_{x\in\mathbb{R}}P\left(\sum_{i=1}^{n}\varepsilon_iv_i\in(x-1,x+1)\right)\leq Cn^{-\frac{1}{2}}
\end{equation}
implies the lower bound of the variance of $\sum_{i=1}^{n}\varepsilon_iv_i$. To be more precise, we have the following proposition.
\begin{prop}\label{prop: var}
Let $(\varepsilon_i)_{i=1,2,\ldots,n}$ be the random spins in an Ising model. Under the assumption \eqref{eq: boundedness assumption} on the finiteness of $K$, if $|v_i|\geq 1$ for $i=1,2,\ldots,n$, then there exists $c(K)>0$ such that
\[\Var(\sum_{i=1}^{n}v_i\varepsilon_i)>c(K)n.\]
\end{prop}
\begin{proof}
By Theorem~\ref{thm: main}, \eqref{eq: LO ineq} holds with $C=C(K)$. Let $F(x)$ be the distribution function of $\sum_{i=1}^{n}\varepsilon_iv_i$. Then the jump $F(x)-F(x-0)$ is at most $C(K)n^{-\frac{1}{2}}\leq 1/10$ for sufficiently large $n$. Hence, there exists $-\infty<a<b<\infty$ such that $F(a)\in [1/10,2/10]$ and $1-F(b)\in[1/10,2/10]$. Then $F(b)-F(a)\geq 6/10$. Since the mass $P(\sum_{i=1}^{n}\varepsilon_iv_i\in[x-1/2,x+1/2])$ of each closed interval $[x-1/2,x+1/2]$ is at most $C(K)n^{-\frac{1}{2}}$, we must have that
\[b-a\geq \frac{6\sqrt{n}}{10C(K)}.\]
Let $Z$ and $\tilde{Z}$ be two independent copies of $\sum_{i=1}^{n}\varepsilon_iv_i$. Then there exists $c(K)>0$ such that
\begin{align*}
2\Var(\sum_{i=1}^{n}\varepsilon_iv_i)&=E[(Z-\tilde{Z})^2]\\
&\geq (b-a)^2P(Z\leq a,\tilde{Z}>b)\\
&=(b-a)^2 F(a)(1-F(b))\\
&>c(K)n.
\end{align*}
\end{proof}
Let $\Sigma=(\Sigma_{ij})_{1\leq i,j\leq n}$ be the correlation matrix of the spins $(\varepsilon_i)_i$, i.e.,
\[\Sigma_{ij}=E[\varepsilon_{i}\varepsilon_{j}]-E[\varepsilon_{i}]E[\varepsilon_{j}].\]
Then we have that
\[\Var(\sum_{i=1}^{n}v_i\varepsilon_i)=\sum_{i,j=1}^{n}v_iv_j\Sigma_{ij}.\]
By Proposition~\ref{prop: var}, there exists $c(K)>0$ such that the smallest eigenvalue of $\Sigma$ is larger than $c(K)/\log n$. We will not provide the details. Instead, we prove the following stronger result:
\begin{prop}\label{prop: covariance matrix}
Under the assumption \eqref{eq: boundedness assumption} on the finiteness of $K$, the smallest eigenvalue of $\Sigma$ is not less than $e^{-K}$.
\end{prop}
\begin{proof}
As we explained at the beginning of Section~\ref{sect: preliminary}, without loss of generality, we assume that the external field $h$ vanishes. Recall the Edwards-Sokal coupling and the notation in Section~\ref{sect: preliminary}. By Lemma~\ref{lem: from dependent Ising to Rademacher series}, we have that
\begin{align*}
\Var(\sum_{i=1}^{n}\varepsilon_iv_i)&=\Var(\sum_{k=1}^{\ell}\eta_kw_k+a)\\
&\geq E\left[\Var\left(\left.\sum_{k=1}^{\ell}\eta_kw_k+a\right|\text{edge configuration }e\right)\right]\\
&= E\left[\sum_{k=1}^{\ell}w_k^2\right],
\end{align*}
where
\[w_k=\sum_{j\in C_k}S(j)v_j.\]
If the inner cluster $C_k$ consists of a single vertex $j$, we have $|w_k|=|v_{j}|$. Hence, we have that
\begin{equation}
\Var(\sum_{i=1}^{n}\varepsilon_iv_i)\geq \sum_{j\in V_n}P(j\text{ is isolated})v_j^2.
\end{equation}
Conditionally on the spin configuration $\sigma$, the distribution of the edge configuration $e$ is a Bernoulli bond percolation. Moreover, we have that
\[P(e_{ij}=1|\sigma)=p_{ij}:=(1-e^{-|J_{ij}|})(1_{J_{ij}>0}1_{\sigma_i=\sigma_j}+1_{J_{ij}<0}1_{\sigma_i\neq\sigma_j}).\]
Hence, we have that
\begin{align*}
P(j\text{ is isolated})&=E[P(j\text{ is isolated}|\sigma)]\\
&=E\left[\prod_{i}(1-p_{ij})\right]\\
&\geq e^{-\sum_{i}|J_{ij}|}\\
&\geq e^{-K}.
\end{align*}
Then we have that
\[\sum_{i,j=1}^{n}v_iv_j\Sigma_{ij}=\Var(\sum_{i=1}^{n}\varepsilon_iv_i)\geq \sum_{j=1}^{n}P(j\text{ is isolated})v_j^2\geq e^{-K}\sum_{j=1}^{n}v_j^2.\]
Hence, the smallest eigenvalue of the real symmetric matrix $\Sigma$ is at least $e^{-K}$.
\end{proof}

\bibliographystyle{amsalpha}
\providecommand{\bysame}{\leavevmode\hbox to3em{\hrulefill}\thinspace}
\providecommand{\MR}{\relax\ifhmode\unskip\space\fi MR }
\providecommand{\MRhref}[2]{%
  \href{http://www.ams.org/mathscinet-getitem?mr=#1}{#2}
}
\providecommand{\href}[2]{#2}

\end{document}